\documentclass[twoside, 10pt]{amsart}
\usepackage{amsfonts, amsthm, amssymb}
\usepackage{enumerate}
\usepackage{amsmath}

\theoremstyle{plain}

\newtheorem*{thm A}{Theorem~A}
\newtheorem*{thm B}{Theorem~B}
\newtheorem*{thm C}{Theorem~B}
\newtheorem*{pro A}{Proposition~A}
\newtheorem*{pro B}{Proposition~B}
\newtheorem*{lem A}{Lemma~A}
\newtheorem*{lem B}{Lemma~B}
\newtheorem*{lem C}{Lemma~C}
\newtheorem*{lem D}{Lemma~D}

\newtheorem*{rem}{Remark}

\newtheorem*{MT1}{Theorem 1}
\newtheorem*{MT2}{Theorem 2}
\newtheorem*{ackn}{Acknowledgments}

\newtheorem{theorem}{Theorem}[section]

\newtheorem{lemma}[theorem]{Lemma}

\newtheorem{remark}[theorem]{Remark}

\def \N{\nabla}

\def \x{\xi}

\def \x{\xi}

\def \NJ{{\bar R}_{N}}

\def \Q{Q^m}
\def \C{\mathcal C}

\begin{document}

\title[Derivatives of normal Jacobi operator]{Derivatives of normal Jacobi operator on real hypersurfaces in the complex quadric}

\vspace{0.2in}
\author[H. Lee J.D. P\'{e}rez and Y.J. Suh]{Hyunjin Lee, Juan de Dios P\'{erez}, and Young Jin Suh}

\address{\newline
Hyunjin Lee
\newline The Research Institute of Real and Complex Manifolds (RIRCM), Kyungpook National University, Daegu 41566, Republic of Korea}
\email{lhjibis@hanmail.net}

\address{\newline
Juan de Dios P\'{e}rez
\newline Departamento de Geometria y Topologia \& IEMATH, Universidad de Granada, 18071 Granada, Spain}
\email{jdperez@ugr.es}

\address{\newline
Young Jin Suh
\newline Department of Mathematics \& RIRCM, Kyungpook National University, Daegu 41566, Republic of Korea}
\email{yjsuh@knu.ac.kr}


\thanks{* This work was supported by grant Proj. No. NRF-2018-R1D1A1B-05040381 and the first author by NRF-2019-R1I1A1A01050300 from National Research Foundation of Korea. Second author by MINECO-FEDER Project MTM 2016-78807-C2-1-P}

\begin{abstract}
In \cite{S 2017}, Suh gave a non-existence theorem for Hopf real hypersurfaces in the complex quadric with parallel normal Jacobi operator. Motivated by this result, in this paper, we introduce some generalized conditions  named $\mathcal C$-parallel or Reeb parallel normal Jacobi operators. By using such weaker parallelisms of normal Jacobi operator, first we can assert a non-existence theorem of Hopf real hypersurfaces with $\mathcal C$-parallel normal Jacobi operator in the complex quadric $Q^{m}$, $m \geq 3$. Next, we prove that a Hopf real hypersurface has Reeb parallel normal Jacobi operator if and only if it has an $\mathfrak A$-isotropic singular normal vector field.
\end{abstract}

\maketitle

\section{Introduction}\label{section 1}
\setcounter{equation}{0}
\renewcommand{\theequation}{1.\arabic{equation}}
\vspace{0.13in}

As an example of Hermitian symmetric space of compact type, we can give the complex quadric ${Q^m}= SO_{m+2}/SO_mSO_2$, which is a complex hypersurface in the complex projective space~${\mathbb C}P^{m+1}$ (see \cite{R1}, \cite{R2}, \cite{BS}, \cite{S3}, \cite{S4}). The complex quadric can also be regarded as a kind of real Grassmann manifold of compact type with rank 2 (see \cite{He}).  Accordingly, the complex quadric ${Q^m}$ admits two important geometric structures, a complex conjugation structure $A$ and a K\"ahler structure $J$, which anti-commute with each other, that is, $AJ=-JA$. Then for $m \geq 3$ the triple $({Q^m},J,g)$ is a Hermitian symmetric space of compact type with rank $2$ and its maximal sectional curvature is equal to $4$ (see Kobayashi and Nomizu \cite{Ko} and Reckziegel \cite{R}).

\vskip 6pt

In addition to the complex structure~$J$ there is another distinguished geometric structure on~$\Q$, namely a parallel rank two vector bundle ${\mathfrak A}$ which contains an $S^1$-bundle of real structures, that is, complex conjugations $A$ on the tangent spaces of $\Q$. The set is denoted by ${\mathfrak A}_{[z]}=\{A_{{\lambda}\bar z}{\vert}\, {\lambda}\in S^1{\subset}{\mathbb C}\}$, $[z] \in {\Q}$, and it is the set of all complex conjugations defined on $\Q$. Then ${\mathfrak A}_{[z]}$ becomes a parallel rank $2$-subbundle of $\text{End}\ T_{[z]}{\Q}$, $[z] \in {\Q}$. This geometric structure determines a maximal ${\mathfrak A}$-invariant subbundle ${\mathcal Q}$ of the tangent bundle $TM$ of a real hypersurface~$M$ in $\Q$.  Here the notion of parallel vector bundle ${\mathfrak A}$ means that $({\bar\nabla}_XA)Y=q(X)JAY$ for any vector fields $X$ and $Y$ on $\Q$,
where  $\bar\nabla$ and ~$q$ denote a connection and a certain $1$-form defined on $T_{[z]}{\Q}$, $[z] \in {\Q}$ respectively (see \cite{BS}).

\vskip 6pt

Recall that a nonzero tangent vector $W \in T_{[z]}{\Q}$ is called singular if it is tangent to more than one maximal flat in $\Q$. There are two types of singular tangent vectors for the complex hyperbolic quadric~$\Q$:
\begin{itemize}
\item If there exists a conjugation $A \in {\mathfrak A}_{[z]}$ such that $W \in V(A)=\{X \in T_{[z]}{Q^m}{\vert}\,AX=X\}$, then $W$ is singular. Such a singular tangent vector is called {\it ${\mathfrak A}$-principal}.
\item If there exist a conjugation $A \in {\mathfrak A}_{[z]}$ and orthonormal vectors $Z_{1}$, $Z_{2} \in V(A)$ such that $W/||W|| = (Z_{1}+JZ_{2})/\sqrt{2}$, then $W$ is singular. Such a singular tangent vector is called \emph{${\mathfrak A}$-isotropic},
where $V(A)=\{X \in T_{[z]}{Q^m}{\vert}\, AX=X\}$ and $JV(A)=\{X \in T_{[z]}{Q^m}{\vert} \, AX=-X\}$ are the $(+1)$-eigenspace and $(-1)$-eigenspace for the involution $A$
on $T_{[z]}{Q^m}$, $[z] \in {Q^m}$.
\end{itemize}

\vskip 6pt

On the other hand, Jacobi fields along geodesics of a given Riemannian manifold $(\bar M,g)$ satisfy a well known differential equation. This equation naturally inspires the so-called Jacobi operator. That is, if $\bar R$ denotes the curvature operator of $\bar M$, and $Z$ is a tangent vector field to $\bar M$, then the Jacobi operator ${\bar R}_Z \in \text{End}(T_{[z]} {\bar M})$ with respect to $Z$ at $[z] \in \bar M$, defined by $({\bar R}_ZY)([z])=({\bar R}(Y,Z)Z)([z])$ for any $Y \in T_{[z]} {\bar M}$, becomes a self adjoint endomorphism of the tangent bundle $T {\bar M}$ of $\bar M$. Thus, each tangent vector field $Z$ to $\bar M$ provides a Jacobi operator ${\bar R}_Z$ with respect to $Z$. In particular, let $M$ be a real hypersurface in $\bar M$ and $N$ be a normal vector field of $M$ in $\bar M$. Then, for the normal vector field~$N$, we can define the Jacobi operator ${\bar R}_{N} \in  \mathrm{End}{T{\bar M}}$, which is said to be the {\it normal Jacobi operator}. Until now, several geometric properties -~parallelism, invariancy, commuting property - of the normal Jacobi operator ${\bar R}_{N}$ for real hypersurfaces in K\"{a}hler manifolds have been studied by many geometers ~(\cite{JKS 2010}, \cite{LS 2017}, \cite{LS 2018}, \cite{MPJS 2014}, \cite{PJS 2007}).

\vskip 6pt

Actually, the normal Jacobi operator~$\bar R_{N}$ of $M$ is said to be {\it parallel} (\ $\mathcal C$-{\it parallel} or {\it Reeb parallel}, respectively) if $\bar R_{N}\in \mathrm{End}TM$ satisfies
$$
\nabla_{X}{\bar R}_{N}=0 \quad (\ \nabla_{\mathcal C}{\bar R}_{N}=0\ \text{or}\  \ \ \nabla_{\xi}{\bar R}_{N}=0, \ \ \mathrm{respectively)}
$$
for any tangent vector field~$X$ on $M$, where $\mathcal C $ denotes the orthogonal distribution of $\mathrm{span}\{\xi\}$ such that $\mathcal C =\{X \in T_{[z]}M\,|\, X \bot \xi, [z] \in M \}$.

\vskip 6pt
In \cite{JKS 2010}, Jeong, Kim and Suh have investigated the parallel normal Jacobi operator of a Hopf real hypersurface in complex two-plane Grassmannians~$G_{2}(\mathbb C^{m+2})=SU_{m+2}/S(U_{m}U_{2})$ and they gave a non-existence theorem.
Here a real hypersurface $M$ is said to be {\it Hopf} if the Reeb vector field~$\xi$ of~$M$ is principal for the shape operator~$S$, that is, $S\xi=g(S\xi, \xi) \xi = \alpha \xi$. In particular, if the Reeb curvature function $\alpha =g(S\xi, \xi)$ identically vanishes, we say that $M$ has {\it a vanishing geodesic Reeb flow}. Otherwise, $M$ has a {\it non-vanishing geodesic Reeb flow}. Moreover, for the complex quadric~$Q^{m}$ Suh~\cite{S 2017} proved the following:
\begin{thm A}
There do not exist any real hypersurfaces in the complex quadric $Q^{m}$, $m \geq 3$, with parallel normal Jacobi operator.
\end{thm A}

\noindent In addition, recently, Lee and Suh~\cite{LS 2017} have generalized such a notion to the recurrent normal Jacobi operator, that is, $({\nabla}_X{\bar R}_{N})Y={\beta}(X){{\bar R}}_{N}Y$ for a certain $1$-form $\beta$ and any vector fields $X,Y$ on $M$ in $Q^{m}$. By using this notion, they gave a non-existence theorem.

\vskip 6pt

Motivated by these results, as a generalization of Theorem A, in this paper we want to give some classifications of Hopf real hypersurfaces in $Q^{m}$ with respect to $\mathcal C$-parallelism and Reeb parallelism for the normal Jacobi operator~$\bar R_{N}$ which are respectively defined by
\begin{equation*}
{\nabla}_{\mathcal C}{\bar R}_{N}=0
\end{equation*}
and
$$
{\nabla}_{\xi}{\bar R}_{N}=0,
$$
where the orthogonal distribution $\mathcal C $ of the Reeb vector field $\xi$ is defined by
$$
\mathcal C =\{X \in T_{[z]}M\,|\, X \bot \xi, [z] \in M \}.
$$

\vskip 6pt

Moreover, in \cite{LS 2018} from the condition of Hopf and $\mathfrak A$-principal unit normal vector field for real hypersurfaces in $\Q$ the authors have derived the notion of contact $S{\phi}+{\phi}S=k{\phi}$ which is introduced in \cite{BS 2015}. Actually, the concept of {\it contact} real hypersurfaces was regarded as a kind of typical characterizations of model spaces of type~$B$ in non-flat complex space forms (see \cite{Kon 1979} and~\cite{Vernon 1987}). Recently, in \cite{BS 2015} and \cite{KS} Berndt-Suh and Klein-Suh respectively gave those kind of characterizations of Type~$B$ in the complex quadric~$Q^{m}=SO_{m+2}/SO_{m}SO_{2}$ and the complex hyperbolic quadric~$Q^{m*}=SO_{2,m}^{0} / SO_{2}SO_{m}$. In particular, by virtue of the result due to Berndt and Suh in \cite{BS 2015}, we introduce the following theorem which was proved by Lee and Suh~\cite{LS 2018}.
\begin{thm B}
Let $M$ be a Hopf real hypersurface in the complex quadric~$Q^{m}$, $m \geq 3$. Then $M$ has an $\mathfrak A$-principal normal vector field in $Q^{m}$ if and only if $M$ is locally congruent to the model space of type $(\mathcal T_{B})$. Here, the model space of type $(\mathcal T_{B})$ means the tube of radius $0 < r < \frac{\pi}{2 \sqrt{2}}$ around the $m$-dimensional sphere~$S^{m}$ which is embedded in $Q^{m}$ as a real form of $Q^{m}$.
\end{thm B}

\begin{rem}
\rm As typical classifications for real hypersurfaces $M$ with $\mathfrak A$-isotropic singular normal vector field in $Q^{m}$, Berndt-Suh have introduced the notion of {\it isometric Reeb flow} on $M$ in $\Q$, $m \geq 3$, and have asserted that $M$ is locally congruent to a tube over a totally geodesic ${\mathbb C}P^k$ in~${Q}^{2k}$, $m=2k$. Then it is known that the normal vector field of this kind of tube is $\mathfrak A$-isotropic (see \cite{BS 2013}, \cite{S1}, \cite{S 2017} and \cite{SHw}).
\end{rem}

\noindent Then motivated by these backgrounds, first we prove the following:
\begin{MT1}\label{Main Theorem 1}
There does not exist any Hopf real hypersurface with $\mathcal C$-parallel normal Jacobi operator in the complex quadric~$Q^{m}$, $m \geq 3$.
\end{MT1}

\vskip 6pt

Next we consider a Hopf real hypersurface with Reeb parallel normal Jacobi operator in~$Q^{m}$. Then by virtue of Theorem~$\rm B$ mentioned above we can assert another theorem as follows:
\begin{MT2}\label{Main Theorem 2}
Let $M$ be a Hopf real hypersurface in the complex quadric $\Q$, $m \geq 3$. Then $M$ has an $\mathfrak A$-isotropic singular normal vector field if and only if $M$ has a Reeb parallel normal Jacobi operator.
\end{MT2}

\vskip 17pt

\section{The complex quadric}\label{section 2}
 \setcounter{equation}{0}
\renewcommand{\theequation}{2.\arabic{equation}}
\vspace{0.13in}
More in detail related to this section we want to recommend  \cite{BS 2013}, \cite{BS 2015}, \cite{K}, \cite{Ko}, \cite{R}, \cite{S1}, \cite{S2}, \cite{S4} and \cite{SHw}. The complex quadric $Q^m$ is the complex hypersurface in ${\mathbb C}P^{m+1}$ which is defined by the equation $z_1^2 + \cdots + z_{m+2}^2 = 0$, where $z_1,\cdots,z_{m+2}$ are homogeneous coordinates on ${\mathbb C}P^{m+1}$. We equip $Q^m$ with the Riemannian metric which is induced from the Fubini Study metric on ${\mathbb C}P^{m+1}$ with constant holomorphic sectional curvature~$4$. The K\"{a}hler structure on ${\mathbb C}P^{m+1}$ induces canonically a K\"{a}hler structure $(J,g)$ on the complex quadric. For a nonzero vector $z \in \mathbb C^{m+2}$ we denote by $[z]$ the complex span of $z$, that is, $[z]=\mathbb C z = \{\lambda z\,|\, \lambda \in \mathbb C\}$. Note that by definition~$[z]$ is a point in $\mathbb C P^{m+1}$. For each $[z] \in Q^m \subset \mathbb C P^{m+1}$ we identify $T_{[z]}{\mathbb C}P^{m+1}$ with the orthogonal complement ${\mathbb C}^{m+2} \ominus {\mathbb C}z$ of ${\mathbb C}z$ in ${\mathbb C}^{m+2}$ (see \cite{Ko}). The tangent space $T_{[z]}Q^m$ can then be identified canonically with the orthogonal complement ${\mathbb C}^{m+2} \ominus ({\mathbb C}z \oplus {\mathbb C}\rho)$ of ${\mathbb C}z \oplus {\mathbb C}\rho$ in ${\mathbb C}^{m+2}$, where $\rho \in \nu_{[z]}Q^m$ is a normal vector of $Q^m$ in ${\mathbb C}P^{m+1}$ at the point $z$.

\vskip 6pt

The complex projective space ${\mathbb C}P^{m+1}$ is a Hermitian symmetric space of the special unitary group $SU_{m+2}$, namely ${\mathbb C}P^{m+1} = SU_{m+2}/S(U_{m+1}U_1)$. We denote by $o = [0,\ldots,0,1] \in {\mathbb C}P^{m+1}$ the fixed point of the action of the stabilizer $S(U_{m+1}U_1)$. The special orthogonal group $SO_{m+2} \subset SU_{m+2}$ acts on ${\mathbb C}P^{m+1}$ with cohomogeneity one. The orbit containing $o$ is a totally geodesic real projective space ${\mathbb R}P^{m+1} \subset {\mathbb C}P^{m+1}$. The second singular orbit of this action is the complex quadric $Q^m = SO_{m+2}/SO_mSO_2$. This homogeneous space model leads to the geometric interpretation of the complex quadric $Q^m$ as the Grassmann manifold $G_2^+({\mathbb R}^{m+2})$ of oriented $2$-planes in ${\mathbb R}^{m+2}$. It also gives a model of $Q^m$ as a Hermitian symmetric space of rank $2$. The complex quadric $Q^1$ is isometric to a sphere $S^2$ with constant curvature, and $Q^2$ is isometric to the Riemannian product of two $2$-spheres with constant curvature. For this reason we will assume $m \geq 3$ from now on.

\vskip 6pt

For a unit normal vector $\rho$ of $Q^m$ at a point $[z] \in Q^m$ we denote by $A = A_\rho$ the shape operator of $Q^m$ in ${\mathbb C}P^{m+1}$ with respect to $\rho$. The shape operator is an involution on the tangent space $T_{[z]}Q^m$ and
$$
T_{[z]}Q^m = V(A_\rho) \oplus JV(A_\rho),
$$
where $V(A_\rho)$ is the $(+1)$-eigenspace of $A_\rho$ and $JV(A_\rho)$ is the $(-1)$-eigenspace of $A_\rho$.  Geometrically this means that the shape operator $A_\rho$ defines a real structure on the complex vector space $T_{[z]}Q^m$, or equivalently, is a complex conjugation on $T_{[z]}Q^m$. Since the real codimension of~$Q^m$ in ${\mathbb C}P^{m+1}$ is $2$, this induces an $S^1$-subbundle ${\mathfrak A}$ of the endomorphism bundle~$\mathrm{End}(TQ^m)$ consisting of complex conjugations. There is a geometric interpretation of these conjugations. The complex quadric~$Q^m$ can be viewed as the complexification of the $m$-dimensional sphere~$S^m$. Through each point $[z] \in Q^m$ there exists a one-parameter family of real forms of $Q^m$ which are isometric to the sphere $S^m$. These real forms are congruent to each other under action of the center $SO_2$ of the isotropy subgroup of $SO_{m+2}$ at $[z]$. The isometric reflection of $Q^m$ in such a real form $S^m$ is an isometry, and the differential at $[z]$ of such a reflection is a conjugation on $T_{[z]}Q^m$. In this way the family ${\mathfrak A}$ of conjugations on $T_{[z]}Q^m$ corresponds to the family of real forms $S^m$ of~$Q^m$ containing $[z]$, and the subspaces $V(A) \subset T_{[z]}Q^m$ correspond to the tangent spaces $T_{[z]}S^m$ of the real forms~$S^m$ of $Q^m$.

\vskip 6pt

The Gauss equation for $Q^m \subset {\mathbb C}P^{m+1}$ implies that the Riemannian curvature tensor $\bar R$ of $Q^m$ can be described in terms of the complex structure $J$ and the complex conjugations $A \in {\mathfrak A}$:
\begin{equation}\label{Riemannian curvature tensor of Q}
\begin{split}
{\bar R}(X,Y)Z &=  g(Y,Z)X - g(X,Z)Y + g(JY,Z)JX - g(JX,Z)JY \\
& \quad  - 2g(JX,Y)JZ  + g(AY,Z)AX \\
& \quad - g(AX,Z)AY + g(JAY,Z)JAX - g(JAX,Z)JAY
\end{split}
\end{equation}
for any vector fields $X,Y$ and $Z$ in $T_{[z]}Q^m$, $[z]{\in}Q^m$
\vskip 6pt

It is well known that for every unit tangent vector $W \in T_{[z]}Q^m$ there exist a conjugation $A \in {\mathfrak A}$ and orthonormal vectors $Z_{1}$, $Z_{2} \in V(A)$ such that
\begin{equation*}
W = \cos (t) Z_{1} + \sin (t) JZ_{2}
\end{equation*}
for some $t \in [0,\pi/4]$  (see \cite{R}). The singular tangent vectors correspond to the values $t = 0$ and $t = \pi/4$. If $0 < t < \pi/4$ then the unique maximal flat containing $W$ is ${\mathbb R}Z_{1} \oplus {\mathbb R}JZ_{2}$.

\vskip 17pt

\section{Real hypersurfaces in $Q^{m}$}\label{section 3}
\setcounter{equation}{0}
\renewcommand{\theequation}{3.\arabic{equation}}
\vspace{0.13in}

Let $M$ be a  real hypersurface in $Q^m$ and denote by $(\phi,\xi,\eta,g)$ the induced almost contact metric structure. By using the Gauss and Weingarten formulas the left-hand side of \eqref{Riemannian curvature tensor of Q} becomes
\begin{equation*}
\begin{split}
{\bar R}(X,Y)Z  & = R(X,Y)Z -g(SY, Z)SX + g(SX, Z)SY \\
& \quad + \big\{g((\nabla_{X}S)Y, Z)- g((\nabla_{Y}S)X, Z) \big \} N,
\end{split}
\end{equation*}
where $R$ and $S$ denote the Riemannian curvature tensor and the shape operator of $M$ in~$Q^m$, respectively.

\vskip 3pt

Note that $JX=\phi X + \eta(X)N$ and $JN=-\xi$, where $\phi X$ is the tangential component of $JX$ and $N$ is a (local) unit normal vector field of $M$. The tangent bundle $TM$ of $M$ splits orthogonally into  $TM = {\mathcal C} \oplus {\mathbb R}\xi$, where ${\mathcal C} = \mathrm{ker}\,\eta =\{ X \in TM \, |\  g(X, \xi) =\eta(X) =0 \}$ is the maximal complex subbundle of $TM$. The structure tensor field $\phi$ restricted to ${\mathcal C}$ coincides with the complex structure~$J$ restricted to~${\mathcal C}$, and $\phi \xi = 0$. Moreover, since the complex quadric $Q^{m}$ has also a real structure~$A$, we decompose $AX$ into its tangential and normal components for a fixed $A \in \mathfrak A_{[z]}$ and $X \in T_{[z]}M$:
\begin{equation}\label{AX}
AX=BX + \rho(X)N
\end{equation}
where $BX$ is the tangential component of $AX$ and
\begin{equation*}
\rho(X)=g(AX, N)=g(X, AN)=g(X, AJ\xi) = g(JX, A \xi).
\end{equation*}
From these notations, taking the tangential and normal components of \eqref{Riemannian curvature tensor of Q}, we obtain
\begin{equation}\label{e: 3.2}
\begin{split}
& R(X,Y)Z - g(SY,Z)SX + g(SX,Z)SY \\
& =  g(Y,Z)X - g(X,Z)Y + g(JY,Z) \phi X - g(JX,Z) \phi Y - 2g(JX,Y) \phi Z \\
& \quad \  + g(AY,Z) BX - g(AX,Z)BY + g(JAY,Z) \phi BX  \\
& \quad \  -  g(JAY,Z) \rho(X) \xi - g(JAX,Z) \phi BY +  g(JAX,Z)\rho(Y) \xi  \\
&=  g(Y,Z)X - g(X,Z)Y + g(\phi Y,Z) \phi X - g(\phi X,Z) \phi Y - 2g(\phi X,Y) \phi Z \\
& \quad \  + g(BY,Z) BX - g(BX,Z)BY +  g(\phi BY, Z) \phi BX - g(\phi BX, Z) \phi BY \\
& \quad \  -  \rho (Y) \eta(Z) \phi BX - \rho (X) g(\phi BY, Z) \xi + \rho (X) \eta(Z) \phi BY + \rho (Y) g(\phi BX, Z) \xi
\end{split}
\end{equation}
and
\begin{equation}\label{e: 3.3}
\begin{split}
(\nabla_{X}S)Y - (\nabla_{Y}S)X & =  \eta(X) \phi Y - \eta(Y) \phi X  - 2 g(\phi X,Y)\xi  - g(X, \phi A \xi)BY \\
& \quad \ + g(\phi A \xi, Y)BX  + g(A\xi, X) \phi BY  +  g(A\xi, X) g(\phi A \xi, Y) \xi \\
& \quad \ -  g(A\xi, Y) \phi BX   - g(A\xi, Y) g(\phi A \xi, X) \xi,
\end{split}
\end{equation}
which are called the equations of Gauss and Codazzi, respectively.

\vskip 6pt

As mentioned in section~\ref{section 2}, since the normal vector field~$N$ belongs to $T_{[z]}Q^{m}$, $[z] \in M$, we can choose $A \in {\mathfrak A}_{[z]}$ such that
\begin{equation*}
N=N_{[z]} = \cos (t) Z_1 + \sin (t) JZ_2
\end{equation*}
for some orthonormal vectors $Z_1$, $Z_2 \in V(A)$ and $0 \leq t \leq \frac{\pi}{4}$ (see Proposition~3 in~\cite{R}). Note that $t$ is a function on $M$. If $t=0$, then $N=Z_{1}\in V(A)$, therefore we see that $N$ becomes an $\mathfrak A$-principal singular tangent vector field. On the other hand, if $t=\frac{\pi}{4}$, then $N= \frac{1}{\sqrt{2}}(Z_{1}+JZ_{2})$. That is, $N$ is an $\mathfrak A$-isotropic singular tangent vector field. In addition, since $\xi = -JN$, we have
\begin{equation}\label{AN, Axi}
\begin{cases}
\xi  =  \sin (t) Z_2 - \cos (t) JZ_1, \\
AN  =  \cos (t) Z_1 - \sin (t) JZ_2,  \\
A\xi  =  \sin (t) Z_2 + \cos (t) JZ_1.
\end{cases}
\end{equation}
This implies $g(\xi,AN) = 0$ and $g(A\xi, \xi) = -g(AN, N)=-\cos ( 2t)$ on $M$. At each point $[z] \in M$ we define the maximal ${\mathfrak A}$-invariant subspace of $T_{[z]} M$, $[z] \in M$, as follows:
\begin{equation*}
{\mathcal Q}_{[z]} = \{X \in T_{[z]}M \mid AX \in T_{[z]}M\ {\rm for\ all}\ A \in {\mathfrak A}_{[z]}\}.
\end{equation*}
It is known that if $N_{[z]}$ is ${\mathfrak A}$-principal, then ${\mathcal Q}_{[z]} = {\mathcal C}_{[z]}$. But if $N_{[z]}$ is not $\mathfrak A$-principal, then ${\mathcal C}_{[z]} = {\mathcal Q}_{[z]} \oplus \mathrm{span}\{AN, A\xi \}$ (see \cite{S1} and \cite{S4}).

\vskip 6pt

We now assume that $M$ is a Hopf real hypersurface in the complex quadric~$Q^{m}$. Then the shape operator~$S$ of~$M$ in $Q^m$ satisfies $S\xi = \alpha \xi$ with the Reeb function $\alpha = g(S\xi,\xi)$ on~$M$. In particular, if $\alpha$ identically vanishes (otherwise, respectively), then it is said  that $M$ has a vanishing (non-vanishing, respectively) geodesic Reeb flow. By virtue of the Codazzi equation~\eqref{e: 3.3}, we obtain the following lemma.
\begin{lemma}[\cite{LS 2017}] \label{lemma Hopf}
Let $M$ be a Hopf real hypersurface in $Q^m$, $m \geq 3$. Then we obtain
\begin{equation}\label{eq: 3.2}
 X \alpha = (\xi \alpha) \eta(X)  + 2g(A\xi,\xi)g(X,AN)
\end{equation}
and
\begin{equation}\label{eq: 3.1}
\begin{split}
2S \phi SX & = \alpha \phi SX + \alpha S\phi X + 2 \phi X + 2 g(X,\phi A \xi) A\xi \\
& \ \  -  2 g(X,A\xi) \phi A \xi  - 2g(\xi,A\xi) g(\phi A \xi, X) \xi + 2 \eta(X) g(\xi,A\xi) \phi A \xi
\end{split}
\end{equation}
for any tangent vector field $X$ on $M$.
\end{lemma}

\begin{remark}\label{remark 3.2}
{\rm By virtue of \eqref{eq: 3.2} we know that {\it if $M$ has either vanishing geodesic Reeb flow or  constant Reeb curvature, then the normal vector $N$ is singular}. In fact, for such real hypersurfaces it follows that \eqref{eq: 3.2} becomes $g(A \xi , \xi) g(X, AN)=0$ for any tangent vector field $X$ on $M$. Since $g(A\xi, \xi)=-\cos (2t)$, $0 \leq t \leq \frac{\pi}{4}$, the case of $g(A\xi,\xi)=0$ implies that the normal vector field~$N$ is $\mathfrak A$-isotropic. On the other hand, if $g(A\xi, \xi)\neq 0$, that is, $g(AN, X)=0$ for all $X \in TM$, then the vector field~$AN  \in TQ^{m}$ is given
$$
AN= \sum_{i=1}^{2m}g(AN, e_{i})e_{i} + g(AN, N)N = g(AN,N)N
$$
for any basis $\{e_{1}, e_{2}, \cdots, e_{2m-1}, e_{2m}=N \, |\, e_{i} \in TM, \ i=1,2, \cdots, 2m-1 \}$ for $TQ^{m}$. By virtue of the property of real structure~$A$ given by $A^{2}=I$, it implies
$$
N=A^{2}N=g(AN, N)AN.
$$
Taking the inner product $N$ of this equation, we get $g(AN, N)=\pm 1$. Since $g(AN, N)=\cos (2t)$ where $t \in [0, \frac{\pi}{4})$, we assert that $g(AN, N)=1$. That is, $AN=N$. Hence $N$ should be $\mathfrak A$-principal.}
\end{remark}

\vskip 3pt

On the other hand, from the property of $g(A\xi, N)=0$ we assert that $A\xi$ is a unit tangent vector field on $M$ in $Q^{m}$. Hence by Gauss formula, ${\bar \nabla}_{X}Y  = \nabla_{X}Y + \sigma(X,Y)$, and $\N_{X}\xi = \phi SX $ for $X$, $Y \in TM$, it induces
\begin{equation*}
\begin{split}
\N_{X}(A\x)& = {\bar \N}_{X}(A\x) - \sigma(X, A\x)  \\
& = q(X) JA\x + A(\N_{X}\x) + g(SX, \xi) AN - g(SX, A \x)N \\
& = q(X) JA\x + A\phi SX + g(SX, \xi) AN - g(SX, A \x)N
\end{split}
\end{equation*}
From $AN=AJ\xi=-JA\xi$ and $JA\xi = \phi A\xi + \eta(A\xi) N$, it gives us
\begin{equation}\label{eq: 3.8}
\left \{
\begin{array}{l}
\mathrm{\scriptstyle Tangential\ Part:} \ \  \nabla_{X}(A \xi)  = q(X) \phi A\xi + B \phi SX - g(SX, \xi) \phi A \xi,  \\
 \\

\mathrm{\scriptstyle Normal \ Part:} \ \  q(X) g(A\xi, \xi)  = - g(AN, \phi SX) + g(SX, \xi) g(A\xi, \xi)  + g (SX, A\xi).
\end{array}
\right.
\end{equation}
In particular, if $M$ is Hopf, then the second equation in \eqref{eq: 3.8} becomes
\begin{equation}\label{eq: 3.5}
q(\x) g(A\xi, \xi) = 2\alpha g(A\xi, \xi).
\end{equation}

\vskip 6pt

Now, we want to introduce the following proposition as a typical characterization of real hypersurfaces in $Q^{m}$ with $\mathfrak A$-principal normal vector field due to Berndt and Suh in \cite{BS 2015}.
\begin{pro B}
Let $(\mathcal T_{B})$ be the tube of radius $0 < r < \frac{\pi}{2 \sqrt{2}}$ around the $m$-dimensional sphere~$S^{m}$ in $Q^{m}$. Then the following hold:
\begin{enumerate}[\rm (i)]
\item {$(\mathcal T_{B})$ is a Hopf real hypersurface.}
\item {The normal bundle of $(\mathcal T_{B})$ consists of $\mathfrak A$-principal singular vector fields.}
\item {$(\mathcal T_{B})$ has three distinct constant principal curvatures.
\begin{center}
\begin{tabular}{l|l|l}
\hline
\mbox{\rm principal curvature} & \mbox{\rm eigenspace}  & \mbox{\rm multiplicity}  \\
\hline
$\alpha = -\sqrt{2}\cot(\sqrt{2}r)$ & $T_{\alpha}=\mathrm{Span}\{\xi\}$ & $1$ \\
$\lambda = \sqrt{2} \tan(\sqrt{2}r)$ & $T_{\lambda}=V(A) \cap {\mathcal C}=\{X \in \mathcal C\,|\, AX=X\}$ & $m-1$\\
$\mu = 0$ & $T_{\mu} = JV(A) \cap {\mathcal C}=\{X \in \mathcal C\,|\, AX=-X\}$ & $m-1$ \\
\hline
\end{tabular}
\end{center}}
\item {$S \phi + \phi S = 2 \delta \phi $, $\delta =-\frac{1}{\alpha}\neq 0$ (contact hypersurface). }
\end{enumerate}
\end{pro B}

\vskip 17pt

\section{$\mathcal C$-parallel normal Jacobi operator}\label{section 4}
\setcounter{equation}{0}
\renewcommand{\theequation}{4.\arabic{equation}}
\vspace{0.13in}
In this section we assume that $M$ is a Hopf real hypersurface in the complex quadric~$Q^{m}$, $m \geq 3$, with $\mathcal C$-parallel normal Jacobi operator, that is,
\begin{equation*}
({\nabla}_{X} {\bar R}_{N})Y = 0
\tag{*}
\end{equation*}
for $X \in \mathcal C$ and $Y \in TM$. Here the distribution~$\mathcal C$ is given by
$\mathcal C = \{ X \in TM\,| \ X \bot \xi \}$.

\vskip 6pt

As mentioned in section~5 in \cite{LS 2017}, the normal Jacobi operator~$\NJ \in \mathrm{End}(TM)$ and its covariant derivative are given for any $Y$, $Z \in TM$ by, respectively,
\begin{equation}\label{e: 4.1}
{\bar R}_{N}Z = Z + 3\eta(Z) \xi - g(A \xi, \xi) BZ - g(\phi A \xi, Z) \phi A \xi - g(A \xi, Z) A \xi,
\end{equation}
and
\begin{equation}\label{e: 4.2}
\begin{split}
&(\nabla_{Y}{\bar R}_{N})Z \\
& = 3 g(Z, \phi SY) \xi + 3 \eta(Z) \phi SY - 2g(A \xi, \phi SY) BZ  \\
& \quad \ \  -q(Y)g(A \xi, \xi) \phi BZ -q(Y)g(A \xi, \xi) g(Z, \phi A \xi ) \xi - g(BZ, SY) \phi A \xi  \\
& \quad \ \ - g(\phi A \xi, Z) BSY - g(BZ, \phi SY) A \xi - g(A \xi, Z) B \phi SY \\
& \quad \ \  + g(\phi A \xi, Z) g(SY, \xi) A \xi + g(A \xi, Z) g(SY, \xi) \phi A \xi,
\end{split}
\end{equation}
where we have used:
$$
g(AN, N) = -g(A \xi, \xi),
$$
$$
AN = AJ\xi = -J A \xi = - \phi A \xi - \eta(A \xi)N,
$$
and
\begin{equation*}
\begin{split}
JAZ& =J(BZ+g(AZ, N)N)  \\
&= \phi BZ + \eta(BZ)N - g(AZ, N) \xi \\
&= \phi BZ + g(Z, \phi A\xi) \xi + \eta(BZ)N.
\end{split}
\end{equation*}

\vskip 6pt

Now, under our assumption we first prove that the normal vector field~$N$ of~$M$ in $Q^{m}$ is singular. By Remark~\ref{remark 3.2}, if the Reeb function~$\alpha=g(S\xi, \xi)$ vanishes identically, then $N$ is singular. Thus we will turn out our problem with respect to the case $\alpha \neq 0$ at the remain part as follows.
\begin{lemma}\label{lemma 4.1}
Let $M$ be a Hopf real hypersurface in the complex quadric~$Q^{m}$, $m \geq 3$, with non-vanishing geodesic Reeb flow. If the normal Jacobi operator ${\bar R}_{N}$ is $\mathcal C$-parallel, then the normal vector field~$N$ is singular.
\end{lemma}

\begin{proof}
From \eqref{e: 4.2}, the condition of $\mathcal C$-parallel normal Jacobi operator gives us
\begin{equation}\label{e: 4.3}
\begin{split}
&(\nabla_{X}{\bar R}_{N})\xi =0 \\
& \Longleftrightarrow  \ \ 3 \phi SX - 2 g(A \xi, \phi SX) A \xi - q(X) g(A\xi, \xi) \phi A \xi \\
& \quad \quad \quad - g(A \xi, SX) \phi A \xi - g(A\xi, \phi SX) A \xi - g(A\xi, \xi) B \phi SX=0
\end{split}
\end{equation}
if $Z=\xi$ and $Y=X \in \mathcal C$. On the other hand, from the property of $JA=-AJ$, we obtain that
\begin{equation*}
\phi BZ + g(\phi A\xi, Z) \xi = - B \phi Z + \eta(Z) \phi A \xi, \ \  Z \in TM.
\end{equation*}
By using this formula, \eqref{e: 4.3} can be rewritten as
\begin{equation*}
\begin{split}
& 3 \phi SX - 2 g(A \xi, \phi SX) A \xi - q(X) g(A\xi, \xi) \phi A \xi - g(A \xi, SX) \phi A \xi \\
& \ \  - g(A\xi, \phi SX) A \xi +  g(A\xi, \xi) \phi B SX + g(A \xi, \xi) g(\phi A \xi, SX)\xi =0.
\end{split}
\end{equation*}
Moreover, from \eqref{eq: 3.8} we obtain $q(X)g(A\xi, \xi) = 2 g(S A \xi, X)$, $X \in \C$. So, the above equation becomes
\begin{equation}\label{e: 4.4}
\begin{split}
& 3 \phi SX + 3 g(S \phi A \xi, X) A \xi - 3g(SA\xi, X) \phi A \xi  \\
& \ \   +  g(A\xi, \xi) \phi B SX + g(A \xi, \xi) g(S \phi A \xi, X)\xi =0
\end{split}
\end{equation}
Taking the inner product of \eqref{e: 4.4} with $\xi$, we get
$$
g(A\xi, \xi) g(S \phi A \xi, X) =0
$$
for all $X \in \C$.

\vskip 3pt

When $g(A \xi, \xi)=0$, the normal vector field~$N$ should be $\mathfrak A$-isotropic. Hence from now on we consider $g(A \xi, \xi) \neq 0$. It implies that $g(S \phi A \xi, X) =0$, which leads to $S \phi A \xi = g(S \phi A \xi, \xi) \xi =0$. From this, \eqref{e: 4.4} becomes
\begin{equation}\label{e: 4.5}
3 \phi SX - 3g(SA\xi, X) \phi A \xi  +  g(A\xi, \xi) \phi B SX =0,
\end{equation}
which yields
\begin{equation*}
\begin{split}
& - 3 SX + 3g(SA\xi, X) A \xi - 3 g(SA\xi, X) g(A\xi, \xi) \xi \\
& \quad  - g(A\xi, \xi) B SX + g(A \xi, \xi) g( A SX, \xi) \xi =0,
\end{split}
\end{equation*}
if we apply the structure tensor~$\phi$ to \eqref{e: 4.5} and use $\phi^{2}Z= -Z + \eta(Z) \xi$, $A\xi = B \xi$. Since $\alpha=g(S \xi, \xi) \neq 0$, we consequently have
\begin{equation}\label{e: 4.6}
\begin{split}
& - 3 \alpha SX + 3 \alpha g(SA\xi, X) A \xi - 3 \alpha g(SA\xi, X) g(A\xi, \xi) \xi \\
& \quad  - \alpha g(A\xi, \xi) B SX + \alpha g(A \xi, \xi) g(X, SA\xi) \xi =0.
\end{split}
\end{equation}

\vskip 3pt

By the way, from \eqref{eq: 3.1} and $S\phi A \xi =0$, we have $
\alpha SA\xi = \beta (\alpha^{2}+ 2 \beta^{2}) \xi - 2 \beta^{2} A \xi$, where $\beta = g(A \xi, \xi)$. Thus it follows, for any $X\in \C$
\begin{equation}\label{e: 4.7}
\alpha g(S A \xi, X) = - 2 \beta^{2}g(A\xi, X).
\end{equation}
So, the equation \eqref{e: 4.6} yields
\begin{equation*}
- 3 \alpha SX - 6 \beta^{2} g(A\xi, X) A \xi + 6 \beta^{3} g(A\xi, X) \xi - \alpha \beta B SX - 2 \beta^{3} g(A \xi, X) \xi =0.
\end{equation*}
Taking the inner product with $A\xi$ of this equation, we get
\begin{equation*}
- 3 \alpha g(SX, A\xi) - 6 \beta^{2} g(A\xi, X)  + 4 \beta^{4} g(A\xi, X) - \alpha \beta g(BSX, A\xi) =0.
\end{equation*}
Since $BA\xi = A^{2}\xi  = \xi$, together with \eqref{e: 4.7} and $\beta \neq 0$, it gives us $g(A\xi, X) =0$ for all $X \in \C$. Hence
\begin{equation}\label{e: 4.8}
A\xi = g(A\xi, \xi) \xi = \beta \xi
\end{equation}
follows.
Applying the real structure~$A$ to this equation, it follows $\xi = A^{2}\xi = \beta A\xi$. By using \eqref{e: 4.8} again, it leads to
$$
\beta^{2}=1.
$$
As stated in Section~\ref{section 3}, we see that $\beta=g(A\xi, \xi)= -\cos (2t)$ where $t \in [0, \frac{\pi}{4})$, since $\beta \neq 0$. Hence $\beta^{2}= \cos^{2}(2t)=1$ implies $t=0$, which means that the normal vector field~$N$ is $\mathfrak A$-principal. Actually, if $t=0$, then $N$ can be expressed as $N= V_{1}$ for some $V_{1} \in V(A)$.

\vskip 3pt

Summing up above discussions, we can give a complete proof of our lemma. 
\end{proof}

\vskip 6pt

By virtue of Remark~\ref{remark 3.2} and Lemma~\ref{lemma 4.1}, if we consider a Hopf real hypersurface $M$ with $\mathcal C$-parallel normal Jacobi operator in $Q^{m}$, $m \geq 3$, then the normal vector field $N$ should be singular, that is, either $\mathfrak A$-principal or $\mathfrak A$-isotropic. Hence let us consider the case of $M$ has $\mathfrak A$-principal normal vector field~$N$ in $Q^{m}$. Then \eqref{e: 4.4} becomes $3 \phi SX - \phi BSX =0$ for all $X \in \C$. Applying the structure tensor~$\phi$ to this equation, we get
\begin{equation}\label{e: 4.9}
- 3 SX + BSX =0
\end{equation}
for $X \in \mathcal C$.

\vskip 3pt

On the other hand, when a Hopf real hypersurface~$M$ in $Q^{m}$ has $\mathfrak A$-principal normal vector field~$N$, it follows that $A \xi = - \xi$ and $AN = N$. In particular, from $AN=N$, we obtain:
\begin{equation*}
BSZ = ASZ - g(ASZ, N)N = ASZ = SZ - 2 \alpha \eta(Z) \xi,
\end{equation*}
where we have used the Weingarten formula, ${\bar \nabla}_{Z}N = -SZ$ for any $Z \in TM$, in the third equality. Thus  \eqref{e: 4.9} gives us $SX=0$ for all $X \in \mathcal C$. That is, the diagonal components of the shape operator~$S$ of $M$ is given
\begin{equation*}
S = \mathrm{diag}(\alpha, \underbrace{0, 0, \cdots, 0}_{2m-1}),
\end{equation*}
which implies the shape operator~$S$ of $M$ should be anti-commuting, $S \phi + \phi S =0$. By virtue of the main theorem in \cite{LS 2018}, we can assert that {\it there does not exists any Hopf real hypersurface with $\mathcal C$-parallel normal Jacobi operator and with $\mathfrak A$-principal normal tangent vector field~$N$ in $Q^{m}$, $m \geq 3$}.

\vskip 6pt

Next, let us assume that $N$ is $\mathfrak A$-isotropic. It implies that $SA\xi = SAN = - S \phi A \xi =0$. So, \eqref{e: 4.4} becomes $\phi SX =0$ for any $X \in \C$. Then applying the structure tensor~$\phi$ to this equation, it gives $S X=0$ for any $X \in \C$. It yields that $S$ satisfies the anti-commuting property, $S \phi + \phi S =0$. So, we obtain that {\it there does not exists any Hopf real hypersurface with $\mathcal C$-parallel normal Jacobi operator and with $\mathfrak A$-isotropic normal tangent vector field~$N$ in $Q^{m}$, $m \geq 3$}, and this completes the proof of our Theorem~\ref{Main Theorem 1}.

\vskip 17pt

\section{Reeb parallel normal Jacobi operator}\label{section 5}
\setcounter{equation}{0}
\renewcommand{\theequation}{5.\arabic{equation}}
\vspace{0.13in}

In this section we assume that $M$ is a Hopf real hypersurface in the complex quadric~$Q^{m}$, $m \geq 3$, with Reeb parallel normal Jacobi operator, that is,
\begin{equation*}
({\nabla}_{\xi} {\bar R}_{N})Y = 0
\tag{**}
\end{equation*}
for all tangent vector fields $Y$ of $M$. Then, the equation~\eqref{e: 4.2} together with \eqref{eq: 3.5} yield
\begin{equation}\label{eq: 4.3}
\begin{split}
& (\nabla_{\xi}{\bar R}_{N})Y = 0  \\
&\ \  \Longleftrightarrow \ \ q(\xi) g(A \xi, \xi) \big \{ \phi BY + g(\phi A \xi, Y) \xi   \big \} =0 \\
& \ \ \Longleftrightarrow \ \ 2 \alpha g(A \xi, \xi) \big \{ \phi BY + g(\phi A \xi, Y) \xi  \big \} =0.
\end{split}
\end{equation}

By virtue of Remark~\ref{remark 3.2}, if the Reeb function~$\alpha=g(S \xi, \xi)$ vanishes identically, then the normal vector field~$N$ of $M$ in $Q^{m}$ should be singular. Hence, now, let us consider the case of $\alpha \neq 0$. From \eqref{eq: 4.3}, we can divide the study into the following two cases.
\vskip 6pt
{\bf Case 1.}\ \  $g(A\xi, \xi)=0$

\vskip 2pt

From the definition of $\mathfrak A$-isotropic singular vector field, the normal vector field~$N$ should be $\mathfrak A$-isotropic.

\vskip 6pt

{\bf Case 2.}\ \  $g(A\xi, \xi) \neq 0$

\vskip 2pt

It implies that $\phi BY + g(\phi A \xi, Y) \xi =0$ for any tangent vector field~$Y \in TM$. Applying the structure tensor~$\phi$, we have
\begin{equation*}
BY = \eta(BY)\xi = g(A\xi, Y) \xi.
\end{equation*}
Since $BY = AY - g(AY, N)N$, it follows
\begin{equation}\label{eq: 4.4}
AY - g(AY, N)N =  g(A\xi, Y) \xi.
\end{equation}
Applying the real structure~$A$ to \eqref{eq: 4.4}, we obtain
\begin{equation}
Y - g(AY, N) AN = g(A\xi, Y) A \xi.
\end{equation}
Since $AN = -\phi A \xi - g(A\xi, \xi) N$, it leads to
\begin{equation*}
Y + g(AY, N) \phi A \xi  + g(AY, N) g(A \xi, \xi) N  = g(A\xi, Y) A \xi.
\end{equation*}
From this, we get
\begin{equation*}
\left \{
\begin{array}{l}
\mathrm{Tangential\ Part:} \ \  Y + g(AY, N) \phi A \xi  = g(A\xi, Y) A \xi   \\
 \\
\mathrm{Normal \ Part:} \ \  g(AY, N) g(A \xi, \xi) =0
\end{array}
\right.
\end{equation*}
Since $g(A\xi, \xi) \neq 0$, the normal part gives us $g(AN, Y) =0$ for all $Y \in TM$. So, it implies that the vector field $AN \in TQ^{m}$ can be expressed by
$$
AN = g(AN, N)N.
$$
It follows thats $N = g(AN, N)AN = \big(g(AN,N) \big)^{2}N$, which implies $g(AN, N) = \pm 1$. On the other hand, as mentioned in section~\ref{section 3}, $g(AN,N)= - \cos(2t)$, $t \in [0, \frac{\pi}{4})$. So, we get $g(AN, N)=1$, when $t=0$. It implies that $N$ should be $\mathfrak A$-principal.

\vskip 6pt

Summing up these observations, we can assert that:
\begin{lemma}\label{lemma 5.1}
Let $M$ be a Hopf real hypersurface in the complex quadric~$Q^{m}$, $m \geq 3$, with Reeb parallel normal Jacobi operator. Then $M$ has singular normal vector field, that is, $N$ is either $\mathfrak A$-principal or $\mathfrak A$-isotropic.
\end{lemma}

\vskip 3pt

Now, we assume that $M$ has an $\mathfrak A$-principal normal vector field in $Q^{m}$. By virtue of Theorem~$\rm B$ given in Introduction, we see that $M$ is locally congruent to a model space of $(\mathcal T_{B})$. Here, the model space of $(\mathcal T_{B})$ is a tube around the $S^{m}$ in $Q^{m}$ with radius~$r \in (0, \frac{\pi}{ 2 \sqrt{2}})$.

\vskip 3pt

But the real hypersurface $(\mathcal T_{B})$ in $\Q$ does not satisfy the property of Reeb parallel normal Jacobi operator. To show this, let us assume that the normal Jacobi operator ${\bar R}_{N}$ of $(\mathcal T_{B})$ is Reeb parallel. It implies
\begin{equation}\label{eq: 4.6}
\alpha \phi BY =0
\end{equation} from \eqref{eq: 4.3}.

\vskip 3pt

On the other hand, if $N$ is $\mathfrak A$-principal, we obtain that $AY \in TM$ for any $Y \in TM$. So, \eqref{eq: 4.6} becomes $\alpha \phi AY =0$ for all $Y \in TM$. On $(\mathcal T_{B})$, the principal curvature $\alpha = -\sqrt{2} \cot (\sqrt{2}r)$ is a non-zero constant function for $r \in (0, \frac{\pi}{ 2 \sqrt{2}})$. So, we consequently have $\phi AY =0$, which implies $AY=\eta(AY) \xi = -\eta(Y)\xi$, together with $A\xi=-\xi$. From the property of $A^{2}=I$, it gives us
$$
Y= -\eta(Y) A\xi =  \eta(Y) \xi
$$
for all $Y \in T (\mathcal T_{B})$, where $T (\mathcal T_{B})$ denotes the tangent space of type $(B)$. Then it yields that $\mathrm{dim}T (\mathcal T_{B}) =1$, which gives us a contradiction. In fact, according to the Proposition~$\rm B$, we see that the dimension of $T (\mathcal T_{B})$ is $2m-1$, that is, $\mathrm{dim} T (\mathcal T_{B}) = 2m-1$. Therefore, it gives $m=1$, which makes a contradiction for $m \geq 3$. Hence it is shown that, {\it if $M$ has an $\mathfrak A$-principal normal vector field, then it does not have Reeb parallel normal Jacobi operator}.

\vskip 6pt

This together with Lemma~\ref{lemma 5.1} that, {\it if $M$ has Reeb parallel normal Jacobi operator, then it has an $\mathfrak A$-isotropic normal vector field}.
%
%

\vskip 17pt

\begin{ackn}
{\rm The present authors would like to express their sincere gratitude to the referee for his/her valuable comments throughout the manuscript. By virtue of his/her efforts we have made a nice version better than the first manuscript.}
\end{ackn}
\vskip 17pt


\end{document}